\documentclass[11pt]{article}

\usepackage{amsmath}
\usepackage{amsthm}
\usepackage{amssymb}
\usepackage{graphicx}
\usepackage{comment}

\usepackage{latexsym,array,amsfonts,eucal}
\usepackage{mathrsfs,dsfont,bbm}
\usepackage{titlesec,fancyhdr,titletoc}
\usepackage{verbatim}

\begin{document}

\title{Comparison for upper tail probabilities of random series}
\author{Fuchang Gao\thanks{fuchang@uidaho.edu, University of Idaho, Moscow, Idaho, USA}
      \and
Zhenxia Liu\thanks{liu3777@vandals.uidaho.edu, University of Idaho, Moscow, Idaho, USA}
      \and
Xiangfeng Yang\thanks{xyang2@tulane.edu, Universidade de Lisboa, Lisboa, Portugal}}

\date{\today}

\maketitle

\begin{abstract}
Let $\{\xi_n\}$ be a sequence of independent and identically distributed random variables. In this paper we study the comparison for two upper tail probabilities $\mathbb{P}\left\{\sum_{n=1}^{\infty}a_n|\xi_n|^p\geq r\right\}$ and $\mathbb{P}\left\{\sum_{n=1}^{\infty}b_n|\xi_n|^p\geq r\right\}$ as $r\rightarrow\infty$ with two different real series $\{a_n\}$ and $\{b_n\}.$ The first result is for Gaussian random variables $\{\xi_n\},$ and in this case these two probabilities are equivalent after suitable scaling. The second result is for more general random variables, thus a weaker form of equivalence (namely, logarithmic level) is proved.
\end{abstract}

\renewcommand{\theequation}{\thesection.\arabic{equation}}
\newtheorem{theorem}{Theorem}[section]
\newtheorem{thm}{Theorem}[section]
\newtheorem{cor}[thm]{Corollary}
\newtheorem{lemma}[thm]{Lemma}
\newtheorem{prop}[thm]{Proposition}
\theoremstyle{definition}
\newtheorem{defn}[thm]{Definition}
\theoremstyle{remark}
\newtheorem{remark}[thm]{Remark}
\numberwithin{equation}{section}
\newcommand{\dx}{\Delta x}
\newcommand{\dt}{\Delta t}

\def\keywords{\vspace{.5em}\hspace{-2em}
{\textit{Keywords and phrases}:\,\relax%
}}
\def\endkeywords{\par}

\def\MSC{\vspace{.0em}\hspace{-2em}
{\textit{AMS 2010 subject classifications}:\,\relax%
}}
\def\endMSC{\par}

\keywords{tail probability, random series, small deviation}

\MSC{60F10; 60G50}

\section{Introduction}
Let $\{\xi_n\}$ be a sequence of independent and identically distributed (i.i.d.) random variables, and $\{a_n\}$ be a sequence of positive real numbers. We consider the random series $\sum_{n=1}^{\infty}a_n\xi_n$. Such random series are basic objects in time series analysis and in regression models (see \cite{Davis-Resnick}), and there have been a lot of research. For example, \cite{Gluskin-Kwapien} and \cite{Latala} studied tail probabilities and moment estimates of the random series when $\{\xi_n\}$ have logarithmically concave tails. Of special interest are the series of positive random variables, or the series of the form $\sum_{n=1}^{\infty}a_n|\xi_n|^p$. Indeed, by Karhunen-Lo\'{e}ve expansion, the $L_2$ norm of a centered continuous Gaussian process $X(t), t\in[0,1],$ can be represented as $\|X\|_{L_2}=\sum_{n=1}^{\infty}\lambda_nZ_n^2$ where $\lambda_n$ are the eigenvalues of the associated covariance
operator, and $Z_n$ are i.i.d. standard Gaussian random variables. It is also known (see \cite{Lifshits-1994}) that the series $\sum_{n=1}^{\infty}a_n|Z_n|^p$ coincides with some bounded  Gaussian process $\{Y_t,t\in \mathrm{T}\}$, where $\mathrm{T}$ is a suitable parameter set: $\sum_{n=1}^{\infty}a_n|Z_n|^p=\sup_{\mathrm{T}}Y_t.$

In this paper, we study the the limiting behavior of the upper tail probability of the series
\begin{align}\label{eq:introduction}
\mathbb{P}\left\{\sum_{n=1}^{\infty}a_n|\xi_n|^p\geq r\right\}\qquad\text{ as }r\rightarrow\infty.
\end{align}
This probability is also called large deviation probability (see \cite{Arcones}). As remarked in \cite{Gao-Li}, for Gaussian process $\|X\|_{L_2}=\sum_{n=1}^{\infty}\lambda_nZ_n^2,$ the eigenvalues $\lambda_n$ are rarely found exactly. Often, one only knows the asymptotic approximation. Thus, a natural question is to study the relation between
the upper tail probability of the original random series and the one with approximated eigenvalues. Also, it is much easier to analyze the rate function in the large deviation theory when $\{a_n\}$ are explicitly given instead of asymptotic approximation.

Throughout this paper, the following notations will be used. The $l^q$ norm of a real sequence $a=\{a_n\}$ is denoted by $||a||_q=\left(\sum_{n=1}^{\infty} a_n^q\right)^{1/q}.$ In particular, the $l^{\infty}$ norm should be understood as $||a||_{\infty}=\max|a_n|.$

We focus on the following two types of comparisons. The first is at the exact level
\begin{align}\label{comparison-exact}
\frac{\mathbb{P}\left\{\sum_{n=1}^{\infty}a_n|\xi_n|\geq r\|a\|_2\beta+|\alpha|\sum_{n=1}^{\infty} a_n\right\}}{\mathbb{P}\left\{\sum_{n=1}^{\infty}b_n|\xi_n|\geq r\|b\|_2\beta+|\alpha|\sum_{n=1}^{\infty} b_n\right\}}\sim1\qquad\text{ as }r\rightarrow\infty
\end{align}
where $\{\xi_n\}$ are i.i.d. Gaussian random variables $N(\alpha,\beta^2);$ see Theorem \ref{exact-standard-normal} and Theorem \ref{exact-general-normal}. This is motivated by \cite{Gao-Hannig-Torcaso} in which the following exact level comparison theorems for small deviations were obtained: as $r\rightarrow0,$
$\mathbb{P}\left\{\sum_{n=1}^{\infty}a_n|\xi_n|\leq r\right\}\sim c\mathbb{P}\left\{\sum_{n=1}^{\infty}b_n|\xi_n|\leq r\right\}$ for i.i.d. random variables $\{\xi_n\}$ whose common distribution satisfies several weak assumptions in the vicinity of zero. The proof of the small deviation comparison is based on the equivalence form of $\mathbb{P}\left\{\sum_{n=1}^{\infty}a_n|\xi_n|\leq r\right\}$ introduced in \cite{Lifshits-1997}. Our proof of upper tail probability comparison (\ref{comparison-exact}) is also based on an equivalent form of $\mathbb{P}\left\{\sum_{n=1}^{\infty}a_n|\xi_n|\geq r\right\}$ in \cite{Lifshits-1994} for Gaussian random variables. The main difficulty is to come up with suitable inequalities which can be used for a specified function $\widehat{\varepsilon}(x,y)$ in Lemma \ref{Lifshits-Theorem2}, and such inequalities are obtained in Lemma \ref{core-lemma-1} and Lemma \ref{core-lemma-2}.

For more general random variables, difficulties arise due to the lack of known equivalent form of $\mathbb{P}\left\{\sum_{n=1}^{\infty}a_n|\xi_n|\geq r\right\}.$ Thus, instead of exact comparison, we consider logarithmic level comparison for upper tail probabilities
\begin{align}\label{comparison-log}
\frac{\log\mathbb{P}\left\{\sum_{n=1}^{\infty}a_n|\xi_n|\geq r\|a\|_q\right\}}{\log\mathbb{P}\left\{\sum_{n=1}^{\infty}b_n|\xi_n|\geq r\|b\|_q\right\}}\sim1\qquad\text{ as }r\rightarrow\infty.
\end{align}
It turns out that under suitable conditions on the sequences $\{a_n\}$ and $\{b_n\}$ the comparison (\ref{comparison-log}) holds true for i.i.d. random variables $\{\xi_n\}$ satisfying $$\lim_{u\rightarrow\infty}u^{-p}\log\mathbb{P}\left\{|\xi_1|\geq u\right\}=-c$$ for some finite constants $p\geq1$ and $c>0;$ see Theorem \ref{rough-comparison}. Here we note that logarithmic level comparisons for small deviation probabilities can be found in \cite{Gao-Li}.

From comparisons (\ref{comparison-exact}) and (\ref{comparison-log}), we see that two upper tail probabilities are equivalent as long as suitable scaling is made. We believe that this holds true for more general random variables; see the conjecture at the end of Section \ref{section-exact-comparison} for details.

\section{Exact comparisons for Gaussian random series}\label{section-exact-comparison}

\subsection{The main results}\label{exact-for-Gaussian}
The following two theorems are the main results in this section. The first one is on standard Gaussian random variables.
\begin{theorem}\label{exact-standard-normal}
Let $\{Z_n\}$ be a sequence of i.i.d. standard Gaussian random variables $N(0,1),$ and $\{a_n\},\{b_n\}$ be two non-increasing sequences of positive real numbers such that $\sum_{n=1}^{\infty}a_n<\infty,\sum_{n=1}^{\infty}b_n<\infty,$
\begin{equation}\label{a-n-b-n-converge}
\begin{aligned}
\prod_{n=1}^{\infty}\left(2-\frac{a_n/\|a\|_2}{b_n/\|b\|_2}\right)\text{ and }\prod_{n=1}^{\infty}\left(2-\frac{b_n/\|b\|_2}{a_n/\|a\|_2}\right)\text{ converge}.
\end{aligned}
\end{equation}
Then as $r\rightarrow\infty$
\begin{align*}
\mathbb{P}\left\{\sum_{n=1}^{\infty}a_n|Z_n|\geq r\|a\|_2\right\}\sim\mathbb{P}\left\{\sum_{n=1}^{\infty}b_n|Z_n|\geq r\|b\|_2\right\}.
\end{align*}
\end{theorem}
For general Gaussian random variables $Z_n,$ it turns out that the condition (\ref{a-n-b-n-converge}) is not convenient to derive the comparison because some more complicated terms appear in the proof. Therefore, an equivalent condition in another form is formulated which forms the following comparison.
\begin{theorem}\label{exact-general-normal}
Let $\{Z_n\}$ be a sequence of i.i.d. Gaussian random variables $N(\alpha,\beta^2),$ and $\{a_n\},\{b_n\}$ be two non-increasing sequences of positive real numbers such that $\sum_{n=1}^{\infty}a_n<\infty,\sum_{n=1}^{\infty}b_n<\infty,$
\begin{equation}\label{a-n-b-n-converge-regular-normal}
\begin{aligned}
\sum_{n=1}^{\infty}\left(1-\frac{a_n/\|a\|_2}{b_n/\|b\|_2}\right)\text{ converges, and }\sum_{n=1}^{\infty}\left(1-\frac{a_n/\|a\|_2}{b_n/\|b\|_2}\right)^2<\infty.
\end{aligned}
\end{equation}
Then as $r\rightarrow\infty$
\begin{align*}
\mathbb{P}&\left\{\sum_{n=1}^{\infty}a_n|Z_n|\geq r\|a\|_2\beta+|\alpha|\sum_{n=1}^{\infty} a_n\right\}\\
&\qquad\qquad\sim\mathbb{P}\left\{\sum_{n=1}^{\infty}b_n|Z_n|\geq r\|b\|_2\beta+|\alpha|\sum_{n=1}^{\infty} b_n\right\}.
\end{align*}
\end{theorem}

\subsection{Proofs of Theorem \ref{exact-standard-normal} and Theorem \ref{exact-general-normal}}
The function $\Phi$ stands for the distribution function of a standard Gaussian random variable
$$\Phi(x)=\int_{-\infty}^x\frac{1}{\sqrt{2\pi}}e^{-u^2/2}du.$$
The first lemma is our starting point.
\begin{lemma}[\cite{Lifshits-1994}]\label{Lifshits-Theorem2} Let $\{\xi_n\}$ be a sequence of i.i.d. Gaussian random variables $N(\alpha,\beta^2),$ and $\{a_n\}$ be a sequence of positive real numbers such that $\sum_{n=1}^{\infty}a_n<\infty.$ Then as $r\rightarrow\infty$
\begin{equation}\label{lemma-Lifshits-Theorem2}
\begin{aligned}
\mathbb{P}&\left\{\sum_{n=1}^{\infty}a_n|\xi_n|\geq r\right\}\\
&\sim\prod_{n=1}^{\infty}\widehat{\varepsilon}\left(\frac{a_n(r-|\alpha|\sum_{n=1}^{\infty}a_n)}{||a||_2^2\beta},\frac{\alpha}{\beta}\right)\cdot\left[1-\Phi\left(\frac{r-|\alpha|\sum_{n=1}^{\infty}a_n}{||a||_2\beta}\right)\right]
\end{aligned}
\end{equation}
where $\widehat{\varepsilon}(x,y)=\Phi(x+|y|)+\exp\{-2x|y|\}\Phi(x-|y|).$
\end{lemma}
\begin{lemma}[Lemma 5 in \cite{Gao-Hannig-Torcaso}]\label{lemma5Gao}
Suppose $\{c_n\}$ is a sequence of real numbers such that $\sum_{n=1}^{\infty}c_n$ converges, and $g$ has total variation $D$ on $[0,\infty).$ Then, for any monotonic non-negative sequence $\{d_n\},$
$$\left|\sum_{n\geq N}c_n g(d_n)\right|\leq (D+\sup_x|g(x)|)\sup_{k>N}\left|\sum_{n=N}^k c_n\right|.$$
\end{lemma}

As mentioned in the introduction, the key step of the proofs is to come up with suitable inequalities that can be used for the function $\widehat{\varepsilon}(x,y)$ in Lemma \ref{Lifshits-Theorem2}. For the proof of Theorem \ref{exact-standard-normal}, we need the following
\begin{lemma}\label{core-lemma-1}
For $a\leq0$ and small enough $\delta,$ we have
$$1+a\cdot \delta\leq (1+\delta)^a.$$
\end{lemma}
The proof of this lemma is trivial. The proof of Theorem \ref{exact-general-normal} requires a more complicated inequality as follows.
\begin{lemma}\label{core-lemma-2}
For a fixed $\sigma>0$ and any $\gamma>0,$ there is a constant $\lambda(\sigma)$ only depending on $\sigma$ such that for any $|a|\leq \sigma$ and $|\delta|\leq \lambda,$
$$1+a\cdot \delta+\gamma\leq (1+\delta)^a(1+\delta^2)(1+\gamma)^2.$$
\end{lemma}
The proof of Lemma \ref{core-lemma-2} is elementary (but not trivial) which is given at the end of this section.

\begin{proof}[Proof of Theorem \ref{exact-standard-normal}]
By otherwise considering $\tilde{a}_n=a_n/\|a\|_2$ and $\tilde{b}_n=b_n/\|b\|_2,$ we assume that $\|a\|_2=\|b\|_2=1.$ It follows from Lemma \ref{Lifshits-Theorem2} that
$$\mathbb{P}\left\{\sum_{n=1}^{\infty}a_n|Z_n|\geq r\right\}\sim\prod_{n=1}^{\infty}2\Phi\left(ra_n\right)\cdot\left[1-\Phi\left(r\right)\right].$$
Therefore,
\begin{align*}
\frac{\mathbb{P}\left\{\sum_{n=1}^{\infty}a_n|Z_n|\geq r\right\}}{\mathbb{P}\left\{\sum_{n=1}^{\infty}b_n|Z_n|\geq r\right\}}\sim\prod_{n=1}^{\infty}\frac{\Phi\left(ra_n\right)}{\Phi\left(rb_n\right)}.
\end{align*}
Now we prove that $\prod_{n=N}^{\infty}\frac{\Phi\left(ra_n\right)}{\Phi\left(rb_n\right)}$ tends to $1$ as $N\rightarrow\infty$ uniformly in $r.$ Then the limit of $\prod_{n=1}^{\infty}\frac{\Phi\left(ra_n\right)}{\Phi\left(rb_n\right)}$ as $r\rightarrow\infty$ is equal to $1$ since the limit of each $\frac{\Phi\left(ra_n\right)}{\Phi\left(rb_n\right)}$ as $r\rightarrow\infty$ is $1.$

By applying Taylor's expansion to $\Phi$ up to the second order, we have
\begin{align*}
\Phi\left(ra_n\right)=&\Phi\left(rb_n\right)+\Phi'\left(rb_n\right)\left(ra_n-rb_n\right)\\
&+\frac{\Phi''(rc_n)}{2}\left(ra_n-rb_n\right)^2
\end{align*}
where $c_n$ is between $a_n$ and $b_n.$ It follows from $\Phi''(rc_n)\leq0$ that
\begin{align*}
\frac{\Phi\left(ra_n\right)}{\Phi\left(rb_n\right)}
\leq1+\frac{rb_n\Phi'\left(rb_n\right)}{\Phi\left(rb_n\right)}\left(\frac{a_n}{b_n}-1\right).
\end{align*}
Let us introduce a new function $g(x)=-\frac{x\Phi'(x)}{\Phi(x)}.$ Now we apply Lemma \ref{core-lemma-1} with $a=g(rb_n)$ to get
\begin{align*}
\frac{\Phi\left(ra_n\right)}{\Phi\left(rb_n\right)}
\leq\left(2-\frac{a_n}{b_n}\right)^{g(rb_n)}.
\end{align*}
It then follows from Lemma \ref{lemma5Gao} that
\begin{align*}
\prod_{n\geq N}\frac{\Phi\left(ra_n\right)}{\Phi\left(rb_n\right)}
&\leq\exp\left\{\sum_{n\geq N}g(rb_n)\log\left(2-\frac{a_n}{b_n}\right)\right\}\\
&\leq\exp\left\{(D+\sup_x|g(x)|)\sup_{k>N}\left|\sum_{n=N}^k\log\left(2-\frac{a_n}{b_n}\right)\right|\right\}\\
\end{align*}
which tends to $1$ uniformly in $r$ from condition (\ref{a-n-b-n-converge}). Thus
\begin{align*}
\limsup_{N\rightarrow\infty}\prod_{n\geq N}\frac{\Phi\left(ra_n\right)}{\Phi\left(rb_n\right)}
\leq 1.
\end{align*}
Similarly,
\begin{align*}
\limsup_{N\rightarrow\infty}\prod_{n\geq N}\frac{\Phi\left(rb_n\right)}{\Phi\left(ra_n\right)}
\leq 1
\end{align*}
which completes the proof.
\end{proof}

\begin{proof}[Proof of Theorem \ref{exact-general-normal}]
From Lemma \ref{Lifshits-Theorem2} we get
\begin{align*}
\frac{\mathbb{P}\left\{\sum_{n=1}^{\infty}a_n|\xi_n|\geq r\|a\|_2\beta+|\alpha|\sum_{n=1}^{\infty} a_n\right\}}{\mathbb{P}\left\{\sum_{n=1}^{\infty}b_n|\xi_n|\geq r\|b\|_2\beta+|\alpha|\sum_{n=1}^{\infty} b_n\right\}}\sim\prod_{n=1}^{\infty}\frac{h(ra_n/\|a\|_2)}{h(rb_n/\|b\|_2)}
\end{align*}
where $h(x)=\Phi(x+|\alpha/\beta|)+\exp\{-2x|\alpha/\beta|\}\Phi(x-|\alpha/\beta|).$ Without loss of generality, we assume $\|a\|_2=\|b\|_2=1.$ We use the notation $f(x)=\exp\{-2x|\alpha/\beta|\}\Phi(x-|\alpha/\beta|),$ thus
$$h(ra_n)=\Phi(ra_n+|\alpha/\beta|)+f(ra_n).$$
Now we apply Taylor's expansions to $\Phi$ at point $rb_n+|\alpha/\beta|$, and to $f$ at point $rb_n$ both up to the second order, so
\begin{align*}
h(ra_n)=&\Phi(rb_n+|\alpha/\beta|)+rb_n\Phi'(rb_n+|\alpha/\beta|)\left(\frac{a_n}{b_n}-1\right)\\
&\qquad\qquad\qquad+\Phi''(rc_{1,n}+|\alpha/\beta|)\left(ra_n-rb_n\right)^2/2\\
&+f(rb_n)+rb_nf'(rb_n)\left(\frac{a_n}{b_n}-1\right)+\frac{r^2b_n^2f''(rc_{2,n})}{2}\left(\frac{a_n}{b_n}-1\right)^2
\end{align*}
where $c_{1,n}$ and $c_{2,n}$ are between $a_n$ and $b_n.$ Because $\Phi''\leq0,$
\begin{align*}
h(ra_n)\leq&h(rb_n)+rb_n\left[\Phi'(rb_n+|\alpha/\beta|)+f'(rb_n)\right]\left(\frac{a_n}{b_n}-1\right)\\
&\qquad\qquad\qquad+\frac{r^2b_n^2f''(rc_{2,n})}{2}\left(\frac{a_n}{b_n}-1\right)^2.
\end{align*}
Taking into account that $\left|r^2b_n^2f''(rc_{2,n})\right|\leq 2c(|\alpha/\beta|)$ for large $N$ uniformly in $r$ with some positive constant $c$ depending on $|\alpha/\beta|,$ we have
\begin{align*}
\frac{h(ra_n)}{h(rb_n)}\leq1+\frac{rb_n\left[\Phi'(rb_n+|\alpha/\beta|)+f'(rb_n)\right]}{h(rb_n)}\left(\frac{a_n}{b_n}-1\right)+c\left(\frac{a_n}{b_n}-1\right)^2.
\end{align*}
The function $g(x):=x\left[\Phi'(x+|\alpha/\beta|)+f'(x)\right]/h(x)$ is bounded and continuously differentiable on $[0,\infty)$ with a bounded derivative. Therefore it follows from Lemma \ref{core-lemma-2} that
\begin{align*}
\frac{h(ra_n)}{h(rb_n)}\leq\left(\frac{a_n}{b_n}\right)^{g(rb_n)}\left(1+\left(\frac{a_n}{b_n}-1\right)^2\right)\left(1+c\left(\frac{a_n}{b_n}-1\right)^2\right)^2.
\end{align*}
By taking the infinite product, we get
\begin{align*}
\prod_{n=N}^{\infty}\frac{h(ra_n)}{h(rb_n)}\leq\prod_{n=N}^{\infty}\left(\frac{a_n}{b_n}\right)^{g(rb_n)}\prod_{n=N}^{\infty}\left(1+\left(\frac{a_n}{b_n}-1\right)^2\right)\left(1+c\left(\frac{a_n}{b_n}-1\right)^2\right)^2.
\end{align*}
According to Lemma \ref{lemma5Gao}, the first product
\begin{align*}
\prod_{n=N}^{\infty}\left(\frac{a_n}{b_n}\right)^{g(rb_n)}&=\exp\left\{\sum_{n\geq N}g(rb_n)\log\left(\frac{a_n}{b_n}\right)\right\}\\
&\leq\exp\left\{(D+\sup_x|g(x)|)\sup_{k>N}\left|\sum_{n=N}^k \log\left(\frac{a_n}{b_n}\right)\right|\right\}
\end{align*}
which tends to $1$ because the series $\sum_{n=1}^{\infty} \log\left(\frac{a_n}{b_n}\right)$ is convergent (this is from condition (\ref{a-n-b-n-converge-regular-normal}), see Appendix for more details).

For the second product, we use $1+x\leq e^x$ to get
\begin{align*}
&\prod_{n=N}^{\infty}\left(1+\left(\frac{a_n}{b_n}-1\right)^2\right)\left(1+c\left(\frac{a_n}{b_n}-1\right)^2\right)^2\\
&\leq\exp\left\{(1+2c)\sum_{n\geq N}\left(\frac{a_n}{b_n}-1\right)^2\right\}
\end{align*}
and this  tends to $1$ because of (\ref{a-n-b-n-converge-regular-normal}). Thus
\begin{align*}
\limsup_{N\rightarrow\infty}\prod_{n=N}^{\infty}\frac{h(ra_n)}{h(rb_n)}\leq 1.
\end{align*}
We can similarly prove $\limsup_{N\rightarrow\infty}\prod_{n=N}^{\infty}\frac{h(rb_n)}{h(ra_n)}\leq 1$ which ends the proof.
\end{proof}

\begin{proof}[Proof of Lemma \ref{core-lemma-2}] We first show that under the assumptions of Lemma \ref{core-lemma-2}, the following inequality holds
\begin{align}\label{middle}
1+a\cdot \delta\leq (1+\delta)^a(1+\delta^2).
\end{align}
Let us consider the function $p(\delta)$ for $|\delta|<1$ and $|a\delta|<1$ defined as
$$p(\delta)=a\log(1+\delta)+\log(1+\delta^2)-\log(1+a\delta).$$
It is clear that $p(0)=0$ and
\begin{align*}
p'(\delta)=\frac{\delta}{(1+\delta)(1+\delta^2)(1+a\delta)}\left[\delta^2\left(a^2+a\right)+\delta\left(2a+2\right)+\left(a^2-a+2\right)\right]
\end{align*}
which is greater than $3/2$ for sufficiently small $\lambda_1$ depending on $\sigma$ with $|a|\leq \sigma$ and $|\delta|\leq \lambda_1,$ since $a^2-a+2\geq 7/4.$ Inequality (\ref{middle}) is thus proved.

Now we define a new function
$$q(\gamma)=(1+\delta)^a(1+\delta^2)(1+\gamma)^2-(1+a\delta+\gamma).$$
From (\ref{middle}) we have $q(0)\geq0.$ Furthermore,
$$q'(\gamma)=(1+\delta)^a(1+\delta^2)2(1+\gamma)-1$$
which can be made positive for small $\lambda_2$ depending on $\sigma$ with $|\delta|\leq \lambda_2.$ The proof is complete by taking $\lambda=\min\{\lambda_1,\lambda_2\}.$
\end{proof}

\subsection{Appropriate extensions}
By using again an equivalence form for $\mathbb{P}\left\{\sum_{n=1}^{\infty}a_n|Z_n|^p\geq r\right\}$ discussed in \cite{Lifshits-1994} with $1\leq p<2,$ we can similarly derive, without much difficulty, exact comparison for the upper tail probabilities of $\sum_{n=1}^{\infty}a_n|Z_n|^p.$ We formulate this as a proposition as follows without a proof.
\begin{prop}\label{exact-general-normal-p}
Let $\{\xi_n\}$ be a sequence of i.i.d. Gaussian random variables $N(\alpha,\beta^2),$ and $\{a_n\},\{b_n\}$ be two sequences of positive real numbers such that $\sum_{n=1}^{\infty}a_n<\infty,\sum_{n=1}^{\infty}b_n<\infty$ and
\begin{equation}\label{abs_sum}
\begin{aligned}
\sum_{n=1}^{\infty}\left|1-\frac{a_n/\sigma_a^p}{b_n/\sigma_b^p}\right|<\infty
\end{aligned}
\end{equation}
for $1\leq p<2,$ $\sigma_a=\left(\sum_{n=1}^{\infty}a_n^{m/p}\right)^{1/m}\beta$ with $m=2p/(2-p).$ Then as $r\rightarrow\infty$
\begin{align*}
\mathbb{P}&\left\{\sum_{n=1}^{\infty}a_n|\xi_n|^p\geq \left(r\sigma_a+|\alpha|\sum_{n=1}^{\infty} a_n^{1/p}\right)^p\right\}\\
&\qquad\qquad\sim\mathbb{P}\left\{\sum_{n=1}^{\infty}b_n|\xi_n|^p\geq \left(r\sigma_b+|\alpha|\sum_{n=1}^{\infty} b_n^{1/p}\right)^p\right\}.
\end{align*}
\end{prop}

Based on what we have observed for Gaussian random variables so far, it is reasonable to believe that after suitable scaling, two upper tail probabilities involving $\{a_n\}$ and $\{b_n\}$ separately are equivalent. Namely, we have the following.

\textbf{Conjecture}: Under suitable conditions on $\{a_n\}$ and $\{b_n\},$ for general i.i.d. random variables $\{\xi_n\},$ the following exact comparison holds
\begin{align*}
\mathbb{P}&\left\{\sum_{n=1}^{\infty}a_n|\xi_n|\geq h\Big(rf^{\xi}(a)+g^{\xi}(a)\Big)\right\}\sim\mathbb{P}\left\{\sum_{n=1}^{\infty}b_n|\xi_n|\geq h\Big(rf^{\xi}(b)+g^{\xi}(b)\Big)\right\}
\end{align*}
for some function $h(r)$ satisfying $\lim_{r\rightarrow\infty}h(r)=\infty,$ and for two suitable scaling coefficients $f^{\xi}(a)$ and $g^{\xi}(a)$ whose values at sequence $a=\{a_n\}$ only depend on $a$ and the structure of the distribution of $\xi_1$ (such as the mean, the variance, the tail behaviors, etc).

In the next section, we show that indeed two upper tail probabilities in the logarithmic level are equivalent after some scaling. This adds more evidence of our conjecture.

\section{Logarithmic level comparison}\label{-section-logarithmic-level-comparison}
In this section, we illustrate the logarithmic level comparison for more general random variables $\{\xi_n\}$ other than the Gaussian ones.
\begin{theorem}\label{rough-comparison}
Let $\{\xi_n\}$ be a sequence of i.i.d. random variables whose common distribution satisfies $\mathbb{E}|\xi_1|<\infty$ and
\begin{align}\label{log-condition-c}
\lim_{u\rightarrow\infty}u^{-p}\log\mathbb{P}\left\{|\xi_1|\geq u\right\}=-c
\end{align}for some constants $p\geq1$ and $0<c<\infty.$ Suppose that a sequence of positive real numbers $\{a_n\}$ is such that $\sum_{n=1}^{\infty}a_n^{2\wedge q}<\infty$ with $q$ given by $\frac{1}{p}+\frac{1}{q}=1.$ Then as $r\rightarrow\infty$
\begin{align}\label{end-rough}
\log\mathbb{P}\left\{\sum_{n=1}^{\infty}a_n|\xi_n|\geq r\right\}\sim -r^p\cdot c\cdot \|a\|_q^{-p}.
\end{align}
\end{theorem}
\begin{remark}
If $\xi_1$ is the standard Gaussian random variable, then $p=2$ and $c=1/2$ in condition (\ref{log-condition-c}). If $\xi_1$ is an exponential random variable with density function $e^{-x}$ on $[0,\infty),$ then $p=c=1.$ One can easily produce more examples. It is straightforward to deduce the following comparison result from (\ref{end-rough}).
\end{remark}

\begin{cor}
Let $\{\xi_n\}$ be a sequence of i.i.d. random variables satisfying the assumptions in Theorem \ref{rough-comparison}. Suppose that two sequences of positive real numbers $\{a_n\}$ and $\{b_n\}$ satisfy $\sum_{n=1}^{\infty}a_n^{2\wedge q}<\infty$ and $\sum_{n=1}^{\infty}b_n^{2\wedge q}<\infty$ with $q$ given by $\frac{1}{p}+\frac{1}{q}=1.$ Then as $r\rightarrow\infty$
$$\frac{\log\mathbb{P}\left\{\sum_{n=1}^{\infty}a_n|\xi_n|\geq r\right\}}{\log\mathbb{P}\left\{\sum_{n=1}^{\infty}b_n|\xi_n|\geq r\right\}}\sim\left(\frac{\|b\|_q}{\|a\|_q}\right)^p$$
and
$$\frac{\log\mathbb{P}\left\{\sum_{n=1}^{\infty}a_n|\xi_n|\geq r\|a\|_q\right\}}{\log\mathbb{P}\left\{\sum_{n=1}^{\infty}b_n|\xi_n|\geq r\|b\|_q\right\}}\sim1.$$
\end{cor}

The proof of Theorem \ref{rough-comparison} is based on the large deviation principle for random series which was derived in \cite{Arcones}. Let us recall a result in \cite{Arcones} (revised a little for our purpose).
\begin{lemma}[\cite{Arcones}]\label{Arones-result}
Let $\{\eta_k\}$ be a sequence of i.i.d. random variables with mean zero satisfying the following condition
\begin{equation}\label{appendix-equation}
\begin{cases}
&\lim_{u\rightarrow\infty}u^{-p}\log\mathbb{P}\left\{\eta_1\leq-u\right\}=-c_1;\\
&\lim_{u\rightarrow\infty}u^{-p}\log\mathbb{P}\left\{\eta_1\geq u\right\}=-c_2,
\end{cases}
\end{equation}
for some $p\geq1$ and $0<c_1,c_2\leq\infty$ with $\min\{c_1,c_2\}<\infty.$ Suppose $\{x_k\}$ is a sequence of real numbers such that $\sum_{k=1}^{\infty}|x_k|^{2\wedge q}<\infty.$ Then the family $\{n^{-1}\sum_{k=1}^{\infty}x_k\eta_k\}$ satisfies the large deviation principle with speed $n^p$ and a rate function
\begin{equation*}
I(z)=\inf\left\{\sum_{j=1}^{\infty}\psi(u_j):\sum_{j=1}^{\infty}u_jx_j=z\right\}, \,\,\,z\in\mathbb{R}
\end{equation*}
where
\begin{align*}
\psi(t)=
\begin{cases}
c_1|t|^p & \text{ if }t<0;\\
0 & \text{ if }t=0;\\
c_2|t|^p & \text{ if }t>0.
\end{cases}
\end{align*}
Namely, for any measurable set $A\subseteq\mathbb{R},$
\begin{align*}
&-\inf\{I(y):y\in\text{interior of }A\}\leq\liminf_{n\rightarrow\infty}n^{-p}\log\mathbb{P}\left\{n^{-1}\sum_{k=1}^{\infty}x_k\eta_k\in A\right\}\\
&\leq\limsup_{n\rightarrow\infty}n^{-p}\log\mathbb{P}\left\{n^{-1}\sum_{k=1}^{\infty}x_k\eta_k\in A\right\}\leq-\inf\{I(y):y\in\text{closure of }A\}.
\end{align*}
\end{lemma}
\begin{proof}[Proof of Theorem \ref{rough-comparison}] We apply Lemma \ref{Arones-result} to the i.i.d. random variables $\eta_k=|\xi_k|-\mathbb{E}|\xi_k|.$ The condition (\ref{log-condition-c}) implies that (\ref{appendix-equation}) is fulfilled. Let us consider a special measurable set $A=[1,\infty).$ By using the Lagrange multiplier, it follows that
\begin{align*}
-\inf\{I(y):y\in\text{interior of }A\}=-\inf\{I(y):y\in\text{closure of }A\}=-c\|x\|_q^{-p}
\end{align*}
(this can be also deduced from Lemma 3.1 of \cite{Arcones}). Then (\ref{end-rough}) follows from the large deviation principle.
\end{proof}

Now let us assume
\begin{align*}
\lim_{u\rightarrow\infty}u^{-p}\log\mathbb{P}\left\{|\xi_1|\geq u\right\}=-c.
\end{align*}
Then it follows easily that
\begin{align*}
\lim_{u\rightarrow\infty}u^{-p/k}\log\mathbb{P}\left\{|\xi_1|^k\geq u\right\}=-c.
\end{align*}
So the logarithmic level comparison for $\xi_n^k$ can be similarly derived as follows.
\begin{prop}\label{rough-comparison-p}
Let $k>0$ be a positive real number, $\{\xi_n\}$ be a sequence of i.i.d. random variables whose common distribution satisfies $\mathbb{E}|\xi_1|^k<\infty$ and
\begin{align*}
\lim_{u\rightarrow\infty}u^{-p}\log\mathbb{P}\left\{|\xi_1|\geq u\right\}=-c
\end{align*}
for some constants $0<c<\infty$ and $p$ such that $p/k\geq1.$ Two sequences of positive real numbers $\{a_n\}$ and $\{b_n\}$ satisfy $\sum_{n=1}^{\infty}a_n^{2\wedge q}<\infty$ and $\sum_{n=1}^{\infty}b_n^{2\wedge q}<\infty$ where $q$ is given by $\frac{1}{p/k}+\frac{1}{q}=1.$ Then as $r\rightarrow\infty$
$$\frac{\log\mathbb{P}\left\{\sum_{n=1}^{\infty}a_n|\xi_n|^k\geq r\right\}}{\log\mathbb{P}\left\{\sum_{n=1}^{\infty}b_n|\xi_n|^k\geq r\right\}}\sim\left(\frac{\|b\|_q}{\|a\|_q}\right)^{p/k}$$
and
$$\frac{\log\mathbb{P}\left\{\sum_{n=1}^{\infty}a_n|\xi_n|^k\geq r\|a\|_q\right\}}{\log\mathbb{P}\left\{\sum_{n=1}^{\infty}b_n|\xi_n|^k\geq r\|b\|_q\right\}}\sim1.$$
\end{prop}

\section*{Appendix}\label{appendix}
In this section, we make a few remarks on the conditions in Theorem \ref{exact-standard-normal} and Theorem \ref{exact-general-normal}. First, we note that conditions (\ref{a-n-b-n-converge}) and (\ref{a-n-b-n-converge-regular-normal}) are not very restrictive, and examples of sequences satisfying these conditions can be produced. For instance, we can consider two sequences with
$$1-\frac{a_n/\|a\|_2}{b_n/\|b\|_2}=\frac{(-1)^n}{n}.$$
To see the relation between (\ref{a-n-b-n-converge}) and (\ref{a-n-b-n-converge-regular-normal}), let us post part of a useful theorem in \cite{Wermuth-1992} from which many convergence results on infinite products and series can be easily derived.
\begin{lemma}[Part (a) of Theorem 1 in \cite{Wermuth-1992}]
Let $\{x_n\}$ be a sequence of real numbers. If any two of the four expressions
$$\prod_{n=1}^{\infty}(1+x_n),\quad \prod_{n=1}^{\infty}(1-x_n),\quad \sum_{n=1}^{\infty}x_n,\quad \sum_{n=1}^{\infty}x_n^2$$
are convergent, then this holds also for the remaining two.
\end{lemma}
Under condition (\ref{a-n-b-n-converge-regular-normal}) in Theorem \ref{exact-general-normal}, it follows from this result that
$$\prod_{n=1}^{\infty}\left(2-\frac{a_n/\|a\|_2}{b_n/\|b\|_2}\right)\text{ and }\prod_{n=1}^{\infty}\frac{a_n/\|a\|_2}{b_n/\|b\|_2}\text{ converge.}$$
This implies that $\sum_{n=1}^{\infty}\log\left(\frac{a_n/\|a\|_2}{b_n/\|b\|_2}\right)$ is convergent. The facts that
$$\sum_{n=1}^{\infty}\left(1-\frac{b_n/\|b\|_2}{a_n/\|a\|_2}\right)^2<\infty\text{ and }\prod_{n=1}^{\infty}\frac{b_n/\|b\|_2}{a_n/\|a\|_2}\text{ converge}$$
yield
$$\prod_{n=1}^{\infty}\left(2-\frac{b_n/\|b\|_2}{a_n/\|a\|_2}\right)\text{ is convergent.}$$


\begin{thebibliography}{00}
\bibitem{Arcones}
M.A.Arcones, The large deviation principle for certain series, ESAIM: Probability and Statistics, 8 (2004) 200-220.

\bibitem{Davis-Resnick}
R.Davis and S.Resnick, Extremes of moving averages of random variables with finite endpoint, Annals of Probability, 19, 1 (1991) 312-328.

\bibitem{Gao-Hannig-Torcaso}
F.Gao, J.Hannig and F.Torcaso, Comparison theorems for small deviations of random series, Electronic J. Probab. 8 (2003) 21:1-17.

\bibitem{Gao-Li}
F.Gao and W.V.Li, Logarithmic level comparison for small deviation probabilities, Journal of Theoretical Probability, 20, 1 (2007) 1-23.

\bibitem{Gluskin-Kwapien}
E.D.Gluskin and S.Kwapie\'{n}, Tail and moment estimates for sums of independent random variables with logarithmically concave tails, Studia Math., 114 (1995) 303-309.

\bibitem{Latala}
R.Latala, Tail and moment estimates for sums of independent random vectors with logarithmically concave tails, Studia Math., 118 (1996) 301-304.

\bibitem{Lifshits-1994}
M.A.Lifshits, Tail probabilities of Gaussian suprema and Laplace transform, Ann. Inst. Henri Poincar\'{e}, 30, 2 (1994) 163-179.

\bibitem{Lifshits-1997}
M.A.Lifshits, On the lower tail probabilities of some random series, Annals of Probbility, 25, 1 (1997) 424-442.

\bibitem{Wermuth-1992}
E.Wermuth, Some elementary properties of infinite products, The American Mathematical Monthly, 99, (1992) 530-537.

\end{thebibliography}
\end{document}